\documentclass[11pt]{article}
\usepackage{latexsym}
\usepackage{amsmath,lipsum}
\usepackage{amssymb,amsfonts,amscd}
\usepackage{graphicx}
\usepackage{multirow}
\usepackage{mathabx} 
\usepackage{hyperref}
\usepackage{color}
\usepackage{subcaption}
\usepackage{cite}
\usepackage{algorithm}
\usepackage{algorithmic}
\usepackage[dvipsnames]{xcolor}
\usepackage{changes}
\usepackage{hhline} 
\usepackage{tabularx}
\usepackage{array}
\newcolumntype{P}[1]{>{\centering\arraybackslash}p{#1}}
\usepackage{siunitx}

\usepackage{diagbox}
\usepackage{tikz}

\newtheorem{theorem}{Theorem}[section]
\newtheorem{proposition}[theorem]{Proposition}
\newtheorem{corollary}[theorem]{Corollary}
\newtheorem{lemma}[theorem]{Lemma}

\newtheorem{rem}[theorem]{Remark}

\newtheorem{example}[theorem]{Example}
\newtheorem{definition}[theorem]{Definition}

\newcommand\qed{{\hspace*{\fill}$\Box$\vskip12pt plus 1pt}}

\newenvironment{proof}{{\noindent\bf Proof.\ }}{\qed}
\newenvironment{remark}{\begin{rem}\em}{\end{rem}}

\newcommand{\Cc}{{\mathbb C}}

\newcommand{\Rr}{{\mathbb R}}
\newcommand{\Zz}{{\mathbb Z}}
\newcommand{\Pp}{{\mathbb P}}

\begin{document}
\markboth{Joseph Cummings and Jonathan D. Hauenstein}
{}

\title{Multi-graded Macaulay Dual Spaces}

\author{\footnotesize Joseph Cummings and Jonathan D. Hauenstein
\footnote{Both authors contributed equally and acknowledge support
from NSF CCF-1812746.}}


\maketitle

\begin{abstract}
  We describe an algorithm for computing Macaulay dual spaces for multi-graded ideals.
  For homogeneous ideals, the natural grading is inherited by the Macaulay dual space which has been leveraged to develop algorithms
  to compute the Macaulay dual space in each homogeneous degree.   Our main theoretical 
  result extends this idea to multi-graded Macaulay
  dual spaces inherited from multi-graded ideals. 
  This natural duality allows ideal operations to be translated from homogeneous ideals to
  their corresponding operations on the multi-graded Macaulay dual spaces.
  In particular, we describe a linear operator with a right inverse for computing
  quotients by a multi-graded polynomial.   
  By using a total ordering on the homogeneous components of the Macaulay dual space, 
  we also describe how to recursively construct a basis for each component. 
  Several examples are included to demonstrate this new approach.
\end{abstract}



\section{Introduction}

For a polynomial system $F \subseteq \Cc[x_1,\dotsc,x_N]$,
many algebraic properties of the ideal $I = \langle F \rangle$ generated by $F$ can,
for example, be deduced from a Gr{\"o}bner basis of $I$,
such as its Hilbert function.
In many instances, one often knows a generating set $F$
for an ideal~$I$, but computing a Gr{\"o}bner basis 
of $I$ could be computationally infeasible. 
From a generating set $F$, another approach
to compute information about the corresponding
ideal is to use Macaulay dual spaces 
which Macaulay formulated as 
inverse systems in \cite{MacaulayBook}
and have been utilized in a variety of scenarios
such as~\cite{NLDT,MultipleZeros,DZ05,GHPS14,
HaoWenrui2013A9Aa,DeflationMult,
LeykinNPD,MM11,ClosednessCondition,StetterBook,ST98,Mourrain-Telen-VanBarel,Batselier-Dreesen-DeMoor}.  
One particular application of interest here
is to compute Hilbert functions of
ideals up to a given degree which are graded by
a finitely generated abelian group~$M$,
called multi-graded~ideals.

Since multi-graded ideals are a generalization of homogeneous ideals, multi-graded Macaulay dual
spaces are a generalization of homogeneous Macaulay
dual spaces.  Moreover, multi-graded ideals naturally arise when considering multi-projective varieties, and more generally, subvarieties of a smooth toric variety. One key 
theoretical result is Thm.~\ref{thm:MhomDualBasis}
which states that the Macaulay dual space of a multi-graded ideal inherits the multi-grading from the ideal.  This is applied in Section~\ref{sec:IdealOperations} to ideal operations
with another key theoretical result being Thm.~\ref{thm: saturation} for
computing ideal quotients using Macaulay dual spaces.
For $\Zz^k$-gradings, an algorithm is described for 
computing each graded piece of the dual space sequentially up to a given degree. 
This is in contrast to other known Gr{\"o}bner basis techniques over semigroup algebras~\cite{Bender-Faugere-Tsigaridas,Faugere-Spaenlehauer-Svartz}.

The rest of the paper is organized as follows. 
Section~\ref{sec:Background} summarizes
necessary background regarding
multi-graded ideals and
Macaulay dual spaces.
Section~\ref{sec: multi-graded macaulay dual spaces} describes multi-graded Macaulay dual spaces which are used in ideal operations
in Section~\ref{sec:IdealOperations}.
Section~\ref{sec:Algorithm} provides an algorithm for computing
multi-graded Macaulay dual spaces and summarizes
a proof-of-concept implementation 
which is used in the examples presented in Section~\ref{sec:Examples}.
A short conclusion is provided in Section~\ref{sec:Conclusion}.

\section{Background}\label{sec:Background}

The following summarizes necessary background
information on multi-graded ideals
and Macaulay dual spaces.

\subsection{Multi-graded ideals}

The first step in defining a multi-graded ideal
is to have a multi-grading on a polynomial ring.

\begin{definition}\label{def: multigradings}
    Let $R = \Cc[x_1,\dotsc,x_N]$ 
    and $M$ be a finitely generated abelian group.
    Then, $R$ is said to be $M$-{\em graded}   
    if there is a direct sum decomposition of the form
    \[
        R = \bigoplus_{m\in M} R_m
    \]
    where $R_{m_1} \cdot R_{m_2} \subseteq R_{m_1 + m_2}$
    for all $m_1,m_2\in M$.
    Moreover, if $m \in M$ and $f \in R_m$, 
    then $f$ is said to be $M$-{\em homogeneous of degree} $m$, denoted $\deg(f) = m$.  
    Finally, if $I \subseteq R$ is an ideal,
    then $I$ is said to be $M$-{\em graded} if $I$
    is generated by $M$-homogeneous polynomials.
\end{definition}

If $\alpha \in (\Zz_{\geq 0})^N$, then consider
\begin{equation}\label{eq:Monomial}
    |\alpha| = \sum_{i=1}^N \alpha_i,
    \quad
    \alpha! = \prod_{i=1}^N \alpha_i!,
    \quad\hbox{and}\quad
    x^\alpha = \prod_{i=1}^N x_i^{\alpha_i}.
\end{equation}
Hence, the standard grading of $R$
is a $\Zz$-grading with
$$R_m = \mathrm{span}_\Cc \left\{x^\alpha~:~
|\alpha|=m, \alpha\in(\Zz_{\geq 0})^N\right\}.$$
The following is an example with a different grading.

\begin{example}\label{ex:Mgraded2vars}
Consider $R = \Cc[x_1,x_2]$ and $M = \Zz$
such that $\deg(x_1) = 2$ and $\deg(x_2) = 1$.
Then, for $m\in\Zz$, 
\begin{equation}\label{eq:Mgraded2vars}
    R_m = \mathrm{span}_\Cc \{x_1^a x_2^{m-2a}~:~m\geq 2a, a\geq \Zz_{\geq 0}\}.
\end{equation}
Hence, $f=x_1 - 3x_2^2 \in R_2$, i.e., $f$ is an $M$-homogeneous polynomial with $\deg(f) = 2$.
\end{example}

Such a construction used in this example
can be naturally generalized to define an $M$-grading
on $R$, namely select \mbox{$m_1,\dotsc,m_N \in M$} and assign $\deg(x_i) = m_i$. 
Thus, for any $\alpha \in (\Zz_{\geq 0})^N$,
\[
    \deg(x^\alpha) = \sum_{i=1}^N \alpha_i m_i.
\]
In particular, for $m\in M$, one has
\[
    R_m = \mathrm{span}_\Cc\left\{x^\alpha ~:~ \deg(x^\alpha) = m, \alpha \in (\Zz_{\geq 0})^N\right\}.
\]

\begin{remark}
If $M = \Zz$ and $\deg(x_1) = \cdots = \deg(x_N) = 1$,
then the $\Zz$-grading on $R$ is the standard grading.
\end{remark}

\begin{example}\label{ex:Hirzebruch}
    Let $M = \Zz^2$, $R = \Cc[x_1,x_2,x_3,x_4]$, and $r \in \Zz_{>0}$. Set 
    $$\begin{array}{lcl}
        \deg(x_1) = (1,-r),&
        \deg(x_2) = (0,1), \\
        \deg(x_3) = (1,0), &
        \deg(x_4) = (0,1).
    \end{array}$$
    This is the Cox ring of the $r^\text{th}$ Hirzebruch surface $\mathcal{H}_r$
    \cite[\S~5.2]{ToricVarietiesCLS}.
    Just as there is a correspondence between homogeneous ideals and projective varieties, there is a correspondence between $M$-graded ideals of $R$ and subvarieties of $\mathcal{H}_r$. 
\end{example}

For an $M$-graded ideal $I$, the multi-graded Hilbert function 
simply records information about the corresponding dimensions of homogeneous components of $R/I$  \cite[\S8.2]{MillerSturmfels}.
In order to have finite dimensions,
we will only consider $M$-gradings in the
remainder of this article such that
$\dim_\Cc(R_m) < \infty$ for all $m \in M$.
This is equivalent to $R_0 = \Cc$,
that is, every polynomial of degree $0$ is~constant.

\begin{definition}
Suppose that $R$ is $M$-graded and $I\subseteq R$ is an 
$M$-graded ideal.
The {\em multi-graded Hilbert function of $I$}
is the function \mbox{$H_{I}:M\rightarrow \Zz$} defined by
$$H_I(m) = \dim_\Cc(R_m) - \dim_\Cc(R_m\cap I).$$
\end{definition}

\begin{example}
Following the setup from Ex.~\ref{ex:Mgraded2vars}
with $I = \langle x_1-3x_2^2\rangle$, one can easily verify
that 
$$H_I(m) = \left\{\begin{array}{lr}
0 & m < 0,\\
1 & m\geq 0.
\end{array}\right.
$$
\end{example}

\begin{example}
An illustration of a grading which will not be considered
is $R = \Cc[x_1,x_2]$ and $M = \Zz$ with
$\deg(x_1) = 1$ and $\deg(x_2) = -1$.  
Thus, for example, $\deg(x_1^kx_2^k) = 0$ 
for any $k\in\Zz_{\geq0}$ so that $\dim_\Cc(R_0) = \infty$.
\end{example}



\subsection{Macaulay dual spaces}

Macaulay dual spaces are a modern form
of inverse systems studied by Macaulay~\cite{MacaulayBook}.
Let $R = \Cc[x_1,\dots,x_N]$, $\alpha\in(\Zz_{\geq0})^N$, 
and $y\in \Cc^N$.
Following \eqref{eq:Monomial}, consider 
the operator $\partial_\alpha[y]: R\rightarrow \Cc$ 
defined by
\[
    \partial_\alpha[y](g) = \left.\frac{1}{\alpha!} \frac{\partial^{|\alpha|}g}{\partial x_1^{\alpha_1} \dots \partial x_N^{\alpha_N}}\right|_{x = y}.
\]
When the context is clear, we will write 
$\partial_\alpha$ instead of $\partial_\alpha[y]$.

\begin{example}
For $R = \Cc[x_1,x_2]$, $\alpha = (3,2)$ and $y = (1,2)$.
Then, 
$$\partial_\alpha[y](x_1^4x_2^3 + 3x_1^3x_2^2 - 2x_1^2 + 3x_2 - 1) = 
\left.\frac{144x_1x_2 + 36}{3! 2!} 
\right|_{x=(1,2)} = 27.$$
In particular, $27$ is the coefficient of $(x_1-1)^3(x_2-2)^2$ 
in a Taylor series expansion of 
$x_1^4x_2^3 + 3x_1^3x_2^2 - 2x_1^2 + 3x_2 - 1$
centered at $y = (1,2)$.
\end{example}

The Macaulay dual space is a $\Cc$-vector space
contained inside of 
$$D_y = \mathrm{span}_\Cc\left\{\partial_\alpha[y]~:~
\alpha\in(\Zz_{\geq0})^N\right\}.$$

\begin{definition}
Let $I\subseteq R$ be an ideal
and $y\in\Cc^N$.
The {\em Macaulay dual space of $I$ at $y$} is
the $\Cc$-vector space 
\[
        D_y(I) = \{ \partial \in D_y ~:~ \partial(g) = 0 \text{ for all } g \in I\}.
\]
\end{definition}

If the dimension of $D_y(I)$ is finite,
then $\dim_{\Cc} D_y(I)$ is the multiplicity
of $y$ with respect to $I$.  
If the dimension of $D_y(I)$ is infinite,
then $y$ is a nonisolated solution 
in $\Cc^N$ to the simultaneous solution
set of $g = 0$ for all $g\in I$.
This fact was exploited in \cite{NLDT}
to develop a numerical local dimension test.

\begin{example}\label{ex:GO}
Let $R = \Cc[x_1,x_2]$
and $I = \langle 29/16 x_1^3 - 2x_1x_2, x_2 - x_1^2\rangle$
arising from the Griewank-Osborne system \cite{GOsystem}.
It is well-known that $y=(0,0)$ has multiplicity $3$
with respect to $I$ and one can easily verify that
\begin{equation}\label{eq:GObasis}
D_0(I) = \mathrm{span}_\Cc \left\{
\partial_{(0,0)}, 
\partial_{(1,0)}, 
\partial_{(0,1)}+\partial_{(2,0)}
\right\}
\end{equation}
is a $3$-dimensional vector space.
\end{example}

From \cite{StetterBook}, for $i=1,\dots,N$,
there are linear anti-differentiation operators \mbox{$\Phi_i : D_y \to D_y$} which are defined via
\begin{equation}\label{eq:Phi}
    \Phi_i(\partial_\alpha) = \begin{cases} \partial_{\alpha - e_i} & \text{if } \alpha_i > 0 \\ 0 &\text{otherwise} \end{cases}
\end{equation}
where $e_i$ is the $i^\text{th}$ standard basis vector. 
From the Leibniz rule, one can easily verify that,
for any $f\in R$ and $\partial \in D_y$,
\begin{equation}\label{eq:Leibniz}
\Phi_i(\partial)(f) = \partial((x_i-y_i) f)
\end{equation}
The following, from \cite{StetterBook,ClosednessCondition},
uses these linear operators
to compute $D_y(I)$
via the so-called {\em closedness subspace condition}
which has been exploited to improve the efficiency
of computing dual spaces~\cite{MultipleZeros,HaoWenrui2013A9Aa}.

\begin{proposition}\label{prop:closedness condition}
    Let $I = \langle f_1,\dotsc,f_t\rangle \subseteq \Cc[x_1,\dotsc,x_N]$, $y \in \Cc^N$, and $\partial \in D_y$. Then, $\partial \in D_y(I)$ if and only if $\partial(f_i) = 0$ for all $1 \leq i \leq t$ and $\Phi_j(\partial) \in D_y(I)$ for all $1 \leq j \leq N$.
\end{proposition}

One key aspect of this closedness condition is that
any basis for the ideal $I$ can be utilized.  

\begin{example}\label{ex:GO2}
Continuing with Ex.~\ref{ex:GO}
with $f_1 = 29/16 x_1^3 - 2x_1x_2$
and $f_2 = x_2-x_1^2$,
consider $\delta = \partial_{(0,1)} + \partial_{(2,0)}$.
Clearly, $\delta(f_1) = 0$
since the monomials $x_1^2$ and $x_2$ do not appear $f_1$.
Next, it is easy to verify that $\delta(f_2) = 1 - 1 = 0$.
Finally, $\Phi_1(\delta) = \partial_{(1,0)}$
and $\Phi_2(\delta) = \partial_{(0,0)}$.
Hence, given that $\partial_{(0,0)}, \partial_{(1,0)} \in D_0(I)$,
Prop.~\ref{prop:closedness condition} allows one to conclude
that $\delta \in D_0(I)$.
\end{example}


\section{Multi-Graded Macaulay Dual Spaces}\label{sec: multi-graded macaulay dual spaces}

For a multi-graded ideal $I\subseteq R=\Cc[x_1,\dots,x_N]$, 
one can consider investigating 
the Macaulay dual space at $y=0\in\Cc^N$ 
to determine properties about $I$.  
The following shows that multi-gradedness of $I$
extends to $D_0(I)$.  

Suppose that $R$ is $M$-graded where the $M$-grading
is induced by assigning $\deg(x_i) = m_i \in M$
such that $\dim_{\Cc}(R_0) = 1$.  In particular, after selecting
a basis of $M$, say $\beta = \{\beta_1,\dots,\beta_k\}$,
one can express each $m_i$ in terms of $\beta$.  
Let $A$ be the $k\times N$ matrix whose $i^{\rm th}$
column corresponds with $m_i$ in terms of $\beta$.
Hence, for any $m\in R$, a basis of monomials for 
$R_m$ is 
$$\{x^\alpha~:~A\cdot \alpha = m, \alpha \in \Zz_{\geq 0}^N\}.$$
In particular, $\dim_\Cc(R_0) = 1$ is equivalent to
$\mathrm{null}~A \cap \Zz_{\geq 0}^N = \{0\}$.

\begin{example}\label{ex:Hirzebruch2}
With the setup from Ex.~\ref{ex:Hirzebruch},
using a standard basis $\beta = \{e_1,e_2\}$ for $M = \Zz^2$,
one has
$$A = \begin{pmatrix} 1 & 0 & 1 & 0 \\
-r & 1 & 0 & 1 \end{pmatrix}
\,\,\,\,\,\,\hbox{and}\,\,\,\,\,\,
N = \begin{pmatrix} 1 & 0 \\ r & -1 \\ -1 & 0 \\ 0 & 1 \end{pmatrix}
$$
where the columns of $N$ span
$\mathrm{null}~A$.
It is clear from the first and third rows of $N$ that $\mathrm{null}~A \cap \Zz_{\geq 0}^4 = \{0\}$.
\end{example}

One can extend the $M$-grading to $D_0$, namely,
for each $m\in M$,
$$D_0^m = \mathrm{span}_\Cc \left\{\partial_\alpha[0]~:~
A\cdot \alpha = m, \alpha \in \Zz_{\geq 0}^N\right\}.$$
Hence, there is a direct sum decomposition of the form
\begin{equation}\label{eq:D0decomp}
D_0 =     
        \bigoplus_{m\in M} D_0^m.
\end{equation}

The following is the key theoretical result
that $D_0(I)$ inherits the multi-grading from $I$.

\begin{theorem}\label{thm:MhomDualBasis}
Suppose that $I\subseteq R$ is an $M$-graded ideal.
Then, the Macaulay dual space $D_0(I)$ is also $M$-graded,
that is,
$$D_0(I) = \bigoplus_{m \in M} D_0^m(I)$$ 
where $D_0^m(I) = D_0^m \cap D_0(I)$.
\end{theorem}
\begin{proof}
Suppose that $\partial \in D_0(I)$. 
Thus, from \eqref{eq:D0decomp}, one can write
$$\partial = \sum_{m\in M} \partial_m$$
where each $\partial_m \in D_0^m$.
The result follows by showing $\partial_m \in D_0(I)$
for all $m\in M$.  To that end, let $g\in I$.
Since $I$ is $M$-graded, one has
$$g = \sum_{m\in M} g_m$$
where each $g_m \in R_m\cap I$.  We claim that,
for any $m\in M$,
\[
    \partial_m(g) = \partial_m(g_m) = \partial(g_m) = 0.
\]
The first equality follows from 
$\partial_m$ being a linear operator
such that $\partial_m(p) = 0$ 
for any $p\in R_q$ for $q\neq m$.
Similarly, the second equality follows from 
linearity along with $\delta(g_m) = 0$ 
for any $\delta \in D_q$ for $q\neq m$.
The last equality follows from 
$g_m\in I$ and $\partial \in D_0(I)$.  
The result now follows since both 
$g\in I$ and $m\in M$ were arbitrary.  
\end{proof}

\begin{example}\label{ex:GO3}
Continuing with the setup from Ex.~\ref{ex:GO},
one can view $I$ as $M$-graded by taking $M = \Zz$
such that $\deg(x_i) = i$, i.e., 
$$A = \begin{pmatrix}
    1 & 2
\end{pmatrix}.$$
Hence, one can interpret 
\eqref{eq:GObasis} as 
$$D_0(I) = D_0^0(I) \oplus D_0^1(I) \oplus D_0^2(I)$$
where 
\begin{equation}\label{eq:GObasis2}
\begin{array}{c}
D_0^0(I) = \mathrm{span}_\Cc\{\partial_{(0,0)}\},
D_0^1(I) = \mathrm{span}_\Cc\{\partial_{(1,0)}\}, \hbox{~and~} \\
D_0^2(I) = \mathrm{span}_\Cc\{\partial_{(0,1)} + \partial_{(2,0)}\}.
\end{array}
\end{equation}
\end{example}

Adapting, for example, the proof of 
\cite[Thm.~3.2]{GHPS14}, the multi-graded Hilbert function 
is simply the dimension of the corresponding
Macaulay dual space.

\begin{proposition}
Suppose that $R$ is $M$-graded and $I\subseteq R$
is an $M$-graded ideal.  Then, for all $m\in M$,
$$H_I(m) = \dim_\Cc(D_0^m(I)).$$
\end{proposition}

\begin{example}\label{ex:GO4}
From Ex.~\ref{ex:GO3}, one has $H_I(0) = H_I(1) = H_I(2) = 1$
and otherwise equal to $0$.
\end{example}

With the $M$-grading on $D_0$, one can view
the anti-differentiation operators $\Phi_i$ 
in \eqref{eq:Phi} as operators 
from $D_0^m$ to $D_0^{m-A\cdot e_i}$
and refine the closedness subspace condition 
in Prop.~\ref{prop:closedness condition}
to~the~multi-graded~case.

\begin{corollary}\label{cor: dSpace membership}
Suppose that $I=\langle f_1,\dots,f_t\rangle \subseteq R$ is an $M$-graded ideal where each $f_i$ is $M$-homogeneous.  
For each $m\in M$, let
$$C_0^m(I) = \left\{\partial \in D_0^m ~:~ \Phi_i(\partial) \in D_0^{m - Ae_i}(I) \text{ for }i = 1,\dotsc,N\right\}.
$$
be the closedness subspace of degree $m$.
Then,
\[
    D_0^m(I) = \left\{\partial \in C_0^m(I) ~:~ \partial(f_i) = 0 \text{ for all } i \text{ such that }\deg(f_i) = m\right\}.
\]
\end{corollary}

The equation-by-equation approach described in
\cite{HaoWenrui2013A9Aa} for computing closedness
subspaces can easily be adapted to this
multi-graded situation.  
Moreover, to compute $C^m_0(I)$, one must have already 
computed $D_0^{m-Ae_i}(I)$ for each $i = 1,\dots,N$. 
There is a natural question about 
which order one has to compute these spaces. 
To answer this, we make the following definition.

\begin{definition}
For an $M$-grading, the {\em weight semigroup of $M$} is
\begin{equation}\label{eq:WeigthSemiGroup}
\omega = \{m \in M ~:~ R_m \neq 0\}.
\end{equation}
The {\em partial ordering induced by $\omega$},
denoted $\preceq_\omega$ is defined by
$$m_1 \preceq_\omega m_2 \quad \Longleftrightarrow \quad m_2-m_1 \in \omega.$$
\end{definition}

Note that $\omega$ is indeed a semigroup and, by our
assumptions on the $M$-grading of $R$, the positive hull of $\omega$, denoted $\omega_\Rr$, in $M\otimes \Rr$ is a pointed polyhedral cone, called the {\em weight cone} of $M$.

\begin{proposition}
$\preceq_\omega$ is a partial ordering on $\omega$.
\end{proposition}
\begin{proof}
Reflexivity follows since $m - m = 0 \in \omega$ so 
that $m \preceq_\omega m$. 
For anti-symmetry, suppose $m_1 \preceq_\omega m_2$ and $m_2 \preceq_\omega m_1$.  Hence, both $m_1 - m_2$ and $-(m_1 - m_2)$ are in $\omega$. 
Our assumptions on the $M$-grading imply that 
$a,-a \in \omega$ if and only if $a = 0$. 
Hence, $m_1 = m_2$.
For transitivity, suppose $m_1 \preceq_\omega m_2$ and $m_2 \preceq_\omega m_3$.  
Then, since $m_2-m_1,m_3-m_2\in \omega$ and since $\omega$ is a semigroup, 
$$m_3-m_1 = (m_3-m_2) + (m_2-m_1)\in \omega.$$
Hence, $m_1 \preceq_\omega m_3$.
\end{proof}

Let $\leq_\omega$ be any linear extension of $\preceq_\omega$.
Thus, $C_0^m(I)$ can be computed from knowing 
$D_0^{s}(I)$ for all $s <_\omega m$ as illustrated next.

\begin{example}\label{ex:Hirzebruch3}
Consider Ex.~\ref{ex:Hirzebruch} with $r = 2$
so that $\deg(x_i)$ is the $i^\text{th}$ column of
    \[
    A = \begin{pmatrix}
            1  & 0 & 1 & 0 \\
            -2 & 1 & 0 & 1
        \end{pmatrix}.
    \]
Let $f = x_3 - x_1 x_2^2$ which is $M$-homogeneous 
with $\deg(f) = (1,0)$.  Suppose that
one aims to compute $C_0^{(1,1)}(I) = D_0^{(1,1)}(I)$
via Cor.~\ref{cor: dSpace membership}
by building up.  
To do this, the first step is to order all the points $v \in \omega$ such that $v \preceq_\omega (1,1)$. There are 8 such points corresponding to the lattice points in 
$\omega_\Rr \cap ((1,1) - \omega)_\Rr$, i.e., the lattice points in the quadrilateral with vertices $(0,0), (0,3), (1,1),$ and $(1,-2)$. The 
following illustrates the 
lattice points and the Hasse diagram of the interval $[(0,0), (1,1)]$.

    \begin{center}
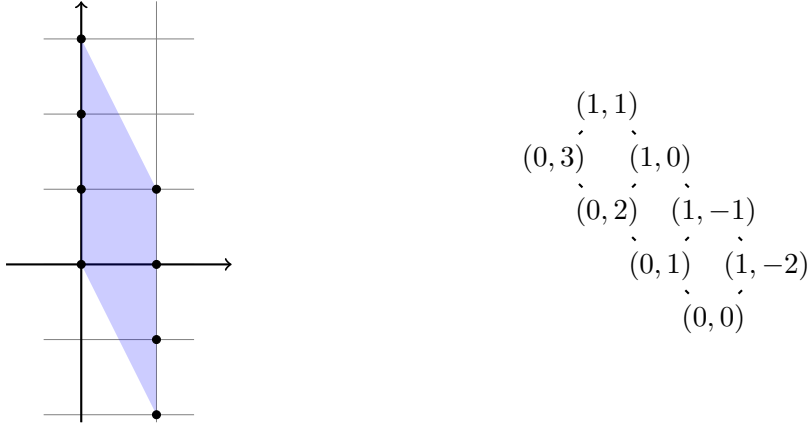
    \begin{minipage}{0.4\textwidth}
    \begin{tikzpicture}
        \draw[help lines, step = 1] (-0.5,-2.1) grid (1.5,3.5);
        \draw[thick, ->] (0,-2.1) to (0,3.5);
        \draw[thick, ->] (-1,0) to (2,0);
        
        \node (L0) at (0,0) {};
        \node (L1) at (0,1) {};
        \node (L2) at (0,2) {};
        \node (L3) at (0,3) {};
        \node (R0) at (1,-2) {};
        \node (R1) at (1,-1) {};
        \node (R2) at (1,0) {};
        \node (R3) at (1,1) {};

        \fill[fill=blue, opacity= 0.2] (0,0)--(1,-2)--(1,1)--(0,3)--cycle;

        \filldraw (L0) circle (1.5pt);
        \filldraw (L1) circle (1.5pt);
        \filldraw (L2) circle (1.5pt);
        \filldraw (L3) circle (1.5pt);
        \filldraw (R0) circle (1.5pt);
        \filldraw (R1) circle (1.5pt);
        \filldraw (R2) circle (1.5pt);
        \filldraw (R3) circle (1.5pt);
        
    \end{tikzpicture}
    \end{minipage}
    \begin{minipage}{0.4\textwidth}
    \begin{tikzpicture}
        \node (v8) at (0,0) {$(1,1)$};
        \node [below left of=v8] (v7) {$(0,3)$};
        \node [below right of=v8] (v6) {$(1,0)$};
        \node [below right of=v7] (v5) {$(0,2)$};
        \node [below right of=v6] (v4) {$(1,-1)$};
        \node [below left of=v4] (v3) {$(0,1)$};
        \node [below right of=v4] (v2) {$(1,-2)$};
        \node [below right of=v3] (v1) {$(0,0)$};

        \draw [thick] (v8) -- (v7);
        \draw [thick] (v8) -- (v6);
        \draw [thick] (v7) -- (v5);
        \draw [thick] (v6) -- (v5);
        \draw [thick] (v6) -- (v4);
        \draw [thick] (v5) -- (v3);
        \draw [thick] (v4) -- (v3);
        \draw [thick] (v4) -- (v2);
        \draw [thick] (v2) -- (v1);
        \draw [thick] (v3) -- (v1);
    \end{tikzpicture}
    \end{minipage}
    \end{center}
    There are 8 linear extensions of the partial order $\preceq_\omega$ and we just need to pick one, say
    \[
    (0,0) <_\omega (1,-2) <_\omega (0,1) <_\omega (1,-1) <_\omega (0,2) 
    \]
    \[
    <_\omega (0,3) <_\omega (1,0) <_\omega (1,1).
    \]
    By Cor.~\ref{cor: dSpace membership}, for every $\alpha <_\omega (1,0)$, we know $D_0^\alpha(I) = D_0^\alpha$ since $I$ has no generators of these degrees and the closedness subspace condition is trivial in this range. Thus, one just needs to compute $D_0^{(1,0)}(I)$ 
    and then lift to $C_0^{(1,1)}(I) = D_0^{(1,1)}(I)$.

    For $(1,0)$, we have that 
    \[
    C_0^{(1,0)}(I) = \mathrm{span}_\Cc 
    \left\{\partial_{(1,2,0,0)}, \partial_{(1,1,0,1)}, \partial_{(1,0,0,2)}, \partial_{(0,0,1,0)}\right\}.
    \]
Imposing the vanishing condition for $f$ yields
    \[
    D_0^{(1,0)}(I) = \mathrm{span}_\Cc \left\{\partial_{(0,0,1,0)} + \partial_{(1,2,0,0)},\partial_{(1,1,0,1)}, \partial_{(1,0,0,2)}\right\}.
    \]

    For $(1,1)$, we have four linear maps in consideration to compute  $C_0^{(1,1)}(I) = D_0^{(1,1)}(I)$. 
    The maps $\Phi_1 : D_0^{(1,1)} \to D_0^{(3,0)}$ and 
    \mbox{$\Phi_3 : D_0^{(1,1)} \to D_0^{(0,1)}$} can safely be ignored since the corresponding Macaulay dual spaces 
    of degrees $(3,0)$ and $(0,1)$ are spanned by all mononomials of their respective degrees
    and thus do not add any restrictions to $C_0^{(1,1)}(I)$. 
    Now, the maps $\Phi_2,\Phi_4 : D_0^{(1,1)} \to D_0^{(1,0)}$ do need to be considered as $D_0^{(1,0)}(I)$ has a non-trivial relation. 
    Hence, $C_0^{(1,1)}(I)=\Phi_2^{-1}(C_0^{(1,0)}(I)) \cap \Phi_4^{-1}(C_0^{(1,0)}(I))$, namely
    \begin{align*}
        C_0^{(1,1)}(I) &= D_0^{(1,1)}(I) \\
        &= \mathrm{span}_\Cc\left\{\begin{array}{c}
                \partial_{(1,3,0,0)} + \partial_{(0,1,1,0)}, \partial_{(1,1,0,2)} \\ 
                \partial_{(1,2,0,1)} + \partial_{(0,0,1,1)},  
                \partial_{(1,0,0,3)}
                \end{array} \right\}.\end{align*}
    In particular, $H_I(1,1) = 4$.
\end{example}

\section{Ideal Operations}\label{sec:IdealOperations}

For an $M$-graded ideal $I\subseteq R = \Cc[x_1,\dots,x_N]$,
there is an expected duality between $I_m$
and $D_0^m(I)$ for every $m\in M$.
In particular, this allows for ideal operations 
to be translated to operations of multi-graded
Macaulay dual spaces as summarized below.

\subsection{Ideal membership test}

The following summarizes testing membership
using a multi-graded
Macaulay dual space.  

\begin{corollary}
If $R$ is $M$-graded, $I\subseteq R$
is an $M$-graded ideal, and $g\in R_m$,
then $g\in I$ if and only if $\partial(g) = 0$
for all $\partial \in D_0^m(I)$.
\end{corollary}
\begin{proof}
The result follows immediately from the definition
of $D_0^m(I)$ and Thm.~\ref{thm:MhomDualBasis}.
\end{proof}

Note that the key to this membership test 
is the multi-grading provided by Thm.~\ref{thm:MhomDualBasis}.
Since this was not included in the 
statement of \cite[Thm.~4.6]{LeykinNPD},
a counter example for that statement
was provided in~\cite[\S.~4]{CounterExIdealMembership},
which is considered next in the multi-graded context.

\begin{example}\label{ex:MembershipTest}
For $R = \Cc[x_1,x_2]$, consider the $M=\Zz$-grading
with $\deg(x_i)=i$.  
The ideal $J = \langle x_2-x_1^2,x_2^2\rangle$
is $M$-graded and $g = x_2\in R_2$.  
It is easy to verify that 
$$D_0(J) = D_0^0(J) \oplus D_0^1(J) \oplus D_0^2(J) \oplus D_0^3(J)$$
where 
\begin{equation}\label{eq:Membershipbasis}
\begin{array}{c}
D_0^0(J) = \mathrm{span}_\Cc\{\partial_{(0,0)}\},
D_0^1(J) = \mathrm{span}_\Cc\{\partial_{(1,0)}\}, \\[0.02in]
D_0^2(J) = \mathrm{span}_\Cc\{\partial_{(0,1)} + \partial_{(2,0)}\}, 
D_0^3(J) = \mathrm{span}_\Cc\{\partial_{(1,1)} + \partial_{(3,0)}\}.
\end{array}
\end{equation}
In particular, using $D_0^2(J)$,
since $\left(\partial_{(0,1)} + \partial_{(2,0)}\right)(g) = 1 \neq 0$,
one can conclude that $g\notin J$.
\end{example}

\subsection{Inclusion, sum, and intersection}

The following considers additional ideal operations.

\begin{corollary}\label{cor: basic ops}
Suppose that $R$ is $M$-graded and $I,J\subseteq R$
are $M$-graded ideals.
\begin{enumerate}
\item $I \subset J$ if and only if $D_0^m(I) \supset D_0^m(J)$ for all $m \in M$.
\item $D_0^m(I+J) = D_0^m(I) \cap D_0^m(J)$ for every $m \in M$.
\item $D_0^m(I \cap J) = D_0^m(I) + D_0^m(J)$ for every $m \in M$.
\end{enumerate}
\end{corollary}
\begin{proof}
The first statement immediately follows from Cor.~\ref{cor: dSpace membership}.

    Since $I,J \subseteq I + J$, we know by the first statement that 
    $$D_0^m(I) \cap D_0^m(J) \supseteq D_0^m(I+J)$$ 
for all $m \in M$.    
On the other hand, if $\partial \in D_0^m(I) \cap D_0^m(J)$ and $f + g \in I + J$, then $\partial(f + g) = \partial(f) + \partial(g) = 0$.  Hence, $\partial \in D_0^m(I + J)$
showing the second statement.

Since $I\cap J \subseteq I,J$, the first statement
implies 
$$D_0^m(I \cap J) \supseteq D_0^m(I) + D_0^m(J)$$
for all $m\in M$. One way to see equality is by 
verifying that they have the same dimension, namely,
for all $m\in M$,
    \begin{align*}
        \dim_\Cc D_0^m(I \cap J) &= H_{I \cap J}(m) \\
                                &= H_I(m) + H_J(m) - H_{I + J}(m) \\
                                &= \dim_\Cc (D_0^m(I)) + \dim_\Cc(D_0^m(J)) \\
                                & \,\,\,\,\,\,\,\,\,\,\,\,\,\, - \dim_\Cc(D_0^m(I) \cap D_0^m(J)) \\
                                &= \dim_\Cc(D_0^m(I) + D_0^m(J)).
    \end{align*}    
\end{proof}

Although the first statement in
Cor.~\ref{cor: basic ops} regarding ideal containment
suggests that one needs to test all $m\in M$,
coupling with Cor.~\ref{cor: dSpace membership}
provides that one only needs to test 
the values of $m\in M$ for which there is a generator
of either $I$ or $J$.

\begin{example}\label{ex:Membership}
Since ideals $I$ from Ex.~\ref{ex:GO}
and $J$ from Ex.~\ref{ex:MembershipTest}
have the same grading, 
one can observe from 
\eqref{eq:GObasis2}
and \eqref{eq:Membershipbasis} that $J\subsetneq I$.
In particular, $D_0^k(I) = D_0^k(J)$ for $k=0,1,2$
and $D_0^3(I) = \{0\} \subsetneq D_0^3(J)$.
\end{example}

\subsection{Ideal quotient}

For $M$-graded ideals $I,J\subseteq R$,
the {\em quotient of $I$ by $J$} is the ideal
$$I:J = \{f \in R:~f\cdot J \subseteq I\}.$$
In particular, if $g\in R$, then
$$I:\langle g\rangle = I:g = \{f\in R~:~f\cdot g \in I\}.$$
Hence, if $J = \langle g_1,\dots,g_t\rangle$, then
$$I:J = \bigcap_{i=1}^t I:g_i$$
so that, for every $m\in M$, Cor.~\ref{cor: basic ops} yields
\begin{equation}\label{eq:DualQuotientSum}
D_0^m(I:J) = D_0^m\left(\bigcap_{i=1}^t I:g_i\right)
= \sum_{i=1}^t D_0^m(I:g_i).
\end{equation}
Thus, one needs to only consider quotients by principal
ideals.  

The maps $\Phi_i$ from \eqref{eq:Phi} map arise
from quotients by variables.  

\begin{proposition}\label{prop:QuotientVariable}
For every $m\in M$ and $i=1,\dots,N$, 
$$\Phi_i(D_0^m(I)) = \Phi_i(D_0^m(I\cap \langle x_i\rangle)) = D_0^{m-Ae_i}(I:x_i).$$
\end{proposition}
\begin{proof}
Let $\partial \in D_0^m(I)$
and $g\in I:x_i$.  Since $x_i g \in I$, \eqref{eq:Leibniz}
yields
$$\Phi_{i}(\partial)(g) = \partial(x_i g) = 0.$$
Hence, $\Phi_i(D_0^m(I))\subseteq D_0^{m-Ae_i}(I:x_i)$.

Let $\delta \in D_0^{m-Ae_i}(I:x_i)$
and $f\in I\cap \langle x_i\rangle$.  Define $h = f/x_i \in I:x_i$.  Consider the linear map $\Psi_i:D_0^{m-Ae_i}\rightarrow
D_0^m$ with $\Psi_i(\partial_\alpha) = \partial_{\alpha+e_i}$.  
Clearly, $\Phi_i\circ\Psi_i$
is the identity map.  Hence,
$$\Psi_i(\delta)(f) = \Psi_i(\delta)(x_i h) =
\Phi_i(\Psi_i(\delta)(h)) = \delta(h) = 0$$
so that 
$D_0^{m-Aei}(I:xi) \subseteq \Phi_i(D_0^m(I\cap \langle x_i\rangle))$.

Finally, suppose $\delta = \Phi_i(\partial) \in \Phi_i(D_0^m(I\cap \langle x_i\rangle))$ and $f\in I$.
Then, 
$$\delta(f) = \Phi_i(\partial)(f) = \partial(x_i f) = 0$$
so that 
$\Phi_i(D_0^m(I\cap \langle x_i\rangle)) \subseteq
\Phi_i(D_0^m(I)).$    
\end{proof}

\begin{example}
Continuing with the setup from Ex.~\ref{ex:GO},
\eqref{eq:GObasis} provides
\begin{align*}
D_0(I:x_1) &= \Phi_1(D_0(I)) = \mathrm{span}_\Cc\left\{0,\partial_{(0,0)},
\partial_{(1,0)}\right\},\\
D_0(I:x_2) &= \Phi_2(D_0(I)) = \mathrm{span}_\Cc\left\{0,0,
\partial_{(0,0)}\right\}.
\end{align*}
Hence, the multiplicity of $0$ with respect
to $I:x_1$ and $I:x_2$ is $2$ and~$1$, respectively.
\end{example}

The key to generalize from quotients by a variable
to quotients by a $M$-homogeneous polynomial $g$
via \eqref{eq:Leibniz} is by defining the linear operator
$\Phi_g:D_0\rightarrow D_0$ via
$$\Phi_g(\partial)(f) = \partial(gf).$$
The Leibniz rule provides
\[
    \Phi_g(\partial_\alpha) = 
    \sum_{\substack{\gamma \\ A\cdot \gamma = \deg g}} \partial_\gamma(g) \partial_{\alpha - \gamma}
\]
which has degree $A\cdot (\alpha-\gamma) = A\cdot \alpha  - \deg g$.

\begin{lemma}\label{lemma:inverse}
For $g\in R_m$, there is a linear function $\Psi_g$
such that $\Phi_g \circ \Psi_g$ is the identity map.
\end{lemma}
\begin{proof}
Let $\prec$ be a lexicographic ordering on $(\Zz_{\geq0})^N$
and write
$$g = \sum_\alpha g_\alpha x^\alpha.$$
Define $\alpha_0 = \mathrm{min}_{\prec} \{\alpha ~:~ g_\alpha \neq 0\}$. For any $\beta,\gamma \in (\Zz_{\geq 0})^{N}$,
define
\[
G(\beta,\gamma) = \begin{cases}
g_{\gamma - \beta} &\text{if }\gamma \preceq \beta \\
0 &\text{otherwise.}
\end{cases}.
\]
Thus, define $\Psi_g(\partial_\beta) = \sum_{\alpha} c_\alpha(\beta) \partial_\alpha$ where
$c_\alpha(\beta)$ is
\[
 \begin{cases}
    \frac{1}{g_{\alpha_0}}\left(\delta(\alpha - \alpha_0, \beta) - \sum_{\gamma \succ \alpha} G(\alpha - \alpha_0,\gamma)c_\gamma(\beta)\right) &\alpha_0 \preceq \alpha \\
    0 &\text{otherwise}
\end{cases}.
\]
where $\delta(\gamma,\beta)$ is Kronecker's delta. Consider the following
\begin{align*}
    \Phi_g(\Psi_g(\partial_\beta)) &= \sum_{\alpha} c_\alpha(\beta) \Phi_g(\partial_\alpha) \\
    &= \sum_{\alpha} c_\alpha(\beta)\sum_{
    \substack{\gamma \\ A\cdot \gamma = \deg g}} \partial_\gamma(g) \partial_{\alpha - \gamma} \\
    &= \sum_{\alpha} c_\alpha(\beta)\sum_{\substack{\gamma \\ A\cdot \gamma = \deg g}} G(\gamma,\alpha) \partial_{\alpha - \gamma} \\
    &= \sum_{\substack{\gamma \\ A\cdot \gamma = \deg g}}\left( \sum_{\alpha} G(\gamma,\alpha) c_\alpha(\beta)\right) \partial_\gamma.
\end{align*}
All that remains is to show 
$\sum_{\alpha} G(\gamma,\alpha) c_\alpha(\beta) = \delta(\gamma,\beta)$.
To that end, we break up the sum as follows.
\begin{align*}
    \sum_{\alpha} G(\gamma,\alpha) c_\alpha(\beta) &= \sum_{\alpha \prec \gamma + \alpha_0} G(\gamma,\alpha) c_\alpha(\beta) \\
    &+ G(\gamma,\gamma + \alpha_0) c_{\gamma+\alpha_0}(\beta) + \sum_{\alpha \succ \gamma + \alpha_0} G(\gamma,\alpha) c_\alpha(\beta).
\end{align*}
Suppose $\alpha \prec \gamma + \alpha_0$. If $\alpha \not \geq \alpha_0$, then $c_\alpha(\beta) = 0$. 
Otherwise, \mbox{$G(\gamma,\alpha) = 0$} by construction of $\alpha_0$. Thus,
\[
\sum_{\substack{\alpha \prec \gamma + \alpha_0}} G(\gamma,\alpha) c_\alpha(\beta) = 0.
\]
The definition of $c_{\gamma + \alpha_0}(\beta)$
finishes the claim due to the following.
$$
G(\gamma,\gamma + \alpha_0) c_{\gamma+\alpha_0}(\beta) = \delta(\gamma,\beta) - \sum_{\alpha \succ \gamma + \alpha_0} G(\gamma,\alpha) c_\alpha(\beta).
$$\end{proof}

With this right inverse, one obtains the following.

\begin{theorem}\label{thm: saturation}
Let $I,J\subseteq R$ be $M$-graded ideals,
$g\in R$ be an $M$-homogeneous polynomial, and $m\in M$.
Suppose that $J = \langle g_1,\dots,g_t\rangle$
where each $g_i$ is $M$-homogeneous.
\begin{enumerate}
    \item $\Phi_g(D_0^{m+\deg g}(I)) = \Phi_g(D_0^{m+\deg g}(I \cap \langle g \rangle)) = D_0^m (I:g)$.
    \item $\sum_{i =1}^t \Phi_{g_i}(D_0^{m+\deg g_i}(I))
    = \sum_{i = 1}^t D_0^m(I:g_i) =     D_0^m(I:J).$
\end{enumerate}
\end{theorem}
\begin{proof}
The first follows a similar approach
as the proof of Prop.~\ref{prop:QuotientVariable}
using Lemma~\ref{lemma:inverse}.
The second follows from the first and~\eqref{eq:DualQuotientSum}.
\end{proof}

\begin{example}\label{ex:Quotient}
Consider computing $J:I$
where $I$ is from Ex.~\ref{ex:GO}
and $J$ is from Ex.~\ref{ex:MembershipTest}.
Let $f_1$ and $f_2$ as in Ex.~\ref{ex:GO2}
be generators for $I$
with the Macaulay dual space for $J$ provided
in \eqref{eq:Membershipbasis}.
Since $\deg f_1 = 3$, one only needs to compute
$$\Phi_{f_1}\left(\partial_{(1,1)}+\partial_{(3,0)}\right) = 
\Phi_{f_1}\left(\partial_{(1,1)}\right) + 
\Phi_{f_1}\left(\partial_{(3,0)}\right)
= -3/16 \partial_{(0,0)}$$
to see that $D_0(J:f_1) = \mathrm{span}_{\Cc}\{\partial_{(0,0)}\}$.
Now, since $\deg f_2 = 2$, we start with
$$\Phi_{f_2}\left(\partial_{(0,1)}+\partial_{(2,0)}\right) =
\partial_{(0,0)}-\partial_{(0,0)}=0
$$
so that $D_0^0(J:f_2) = \{0\}$.  For completeness,
one can verify that
$$\Phi_{f_2}\left(\partial_{(1,1)}+\partial_{(3,0)}\right) =
\partial_{(1,0)}-\partial_{(1,0)}=0.
$$
Hence, $D_0(J:f_2) = \{0\}$ which was expected
since $f_2\in J$ yields $J:f_2 = \langle 1 \rangle$.
Therefore,
$$D_0(J:I) = \mathrm{span}_{\Cc}\{\partial_{(0,0)}\}$$
which corresponds with $J:I = \langle x_1, x_2\rangle$.
\end{example}

One can repeatedly compute ideal quotients,
say $I:J$, $(I:J):J$, $((I:J):J):J$, \dots,
which are denoted $I:J$, $I:J^2$, $I:J^3$, \dots.
This sequence stabilizes after finitely many
terms and is equal to the {\em saturation of $I$ with respect to $J$}, namely
$$I : J^\infty = \{f\in R ~:~ f\cdot J^n \subseteq I \text{ for some }n\geq 1\}.$$
In particular, $I : J^p = I : J^{p+1}$ if and only if $I:J^p = I:J^\infty$.
Saturation is useful, for example, to compute
information regarding a non-homogeneous ideal
by homogenizing and saturating away the component
at infinity.

\section{Algorithm and Software}\label{sec:Algorithm}

The results from Sec.~\ref{sec: multi-graded macaulay dual spaces} 
and~\ref{sec:IdealOperations}
lead to algorithms for computing
multi-graded Macaulay dual spaces
as summarized in the following.
Our proof-of-concept implementation
using {\tt Macaulay2}~\cite{M2} 
is available at
\url{https://doi.org/10.7274/j098z894548}.

In an effort to simplify our procedures
and implementation, 
we assume 
that the multi-grading is a 
\mbox{$\Zz^k$-grading}. Moreover, we assume that that, for every $m \in \Zz^k$, 
$R_m$, the $m$-graded component of 
$R = \Cc[x_1,\dotsc,x_N]$ is a finite dimensional complex vector space. Additionally, we will only consider gradings that arise from a matrix $A \in \Zz^{k\times N}$ where $\deg(x_i)$ is the $i^\text{th}$ column of $A$. 
Since it useful for us to have a 
half-space description of 
the weight semigroup $\omega$
as defined in \eqref{eq:WeigthSemiGroup},
we assume there is a matrix $B \in \Zz^{p \times k}$ where the rows are the normal vectors of the half-spaces whose intersection is the weight cone $\omega_\Rr$, so
\[
    \omega_\Rr = \{y \in \Rr^k ~:~ By \geq 0\}
\]
and $\omega = \omega_\Rr \cap \Zz^k$,
i.e., $\omega$ is \textit{saturated}.
Given $m \in \Zz^k$, this enables one to quickly ascertain whether or not $m$ is contained in $\omega$ or not.

As stated in Sec.~\ref{sec: multi-graded macaulay dual spaces}, in order to compute $D_0^m(I)$ for some $m \in \Zz^k$, we first fix a total ordering on the elements on the set 
$$\omega_m = \{s \in \omega ~:~ s 
\preceq_\omega m\}.$$
Since $\omega$ is a saturated semigroup, the set $\omega_m$ can be realized as the lattice points in a polyhedron,
e.g., see Ex.~\ref{ex:Hirzebruch3}.
Hence, a lattice point $s \in \omega$ is less than $m$ in the partial order if and only if $B(m - s) \geq 0$ and $Bs \geq 0$.
Therefore, the set $\omega_m$ is exactly
\[
    \omega_m = \{s \in \Zz^k ~:~ Bm \geq Bs \geq 0\}.
\]

Our first procedure below details how to find a linear extension of the partial order $\prec_\omega$ on $\omega_m$. 
Note that the most expensive part of this procedure is in computing the lattice points in~$\omega_m$. Our implementation used methods from the {\tt Polyhedra} package \cite{PolyhedraSource}. 

\vspace{0.4cm}
\hrule

\begin{description}
    \item[Procedure $\mathrm{SortLatticePoints}(A,B,m)$]
    \item[Input] The matrix $A\in\Zz^{k\times N}$ where the $i^\text{th}$ column is $\deg(x_i)$. The matrix $B\in \Zz^{p\times k}$ where $\omega_\Rr = \{y ~:~ By \geq 0\}$. A lattice point $m \in \omega$.
    \item[Output] A total ordering of the lattice points in $\omega$ less than $m$ in the partial ordering.
    \item[Begin] \hfill 
    
    \begin{enumerate}
        \item Let $\mathrm{\textbf{Unsorted}} := \{s \in \omega ~:~ s \preceq_\omega m\} \setminus \{0\}$ be the non-zero lattice points in $\omega$ less than $m$ in the partial order. This list is the set of non-zero integral solutions, $s$, to the system of inequalities $Bm \geq Bs \geq 0$. Let $\mathrm{\textbf{Sorted}} := \{0\}$
        \item For every $s \in \mathrm{\textbf{Unsorted}}$, check for every $i = 1, \dotsc, N$ if
        \begin{enumerate}
            \item $s - Ae_i \in \mathrm{\textbf{Sorted}}$ or 
            \item $B(s - Ae_i) \not \geq 0$, so $s - Ae_i \notin \omega$.
        \end{enumerate}
        \item If one of (2a) or (2b) is true for every $i$, then add $s$ to \textbf{Sorted} and delete it from \textbf{Unsorted}.
        \item Repeat steps 2 and 3 until \textbf{Unsorted} is empty.
    \end{enumerate}
    \item[Return] \textbf{Sorted}
\end{description}
\hrule 
\vspace{0.4cm}

The following show the correctness of this procedure. 

\begin{lemma}\label{lemma:finite}
    The set $\omega_m$ is finite.
\end{lemma}
\begin{proof}
     Note that $\omega_m$ is the set of lattice points in the polyhedron $\omega_\Rr \cap (m - \omega_\Rr)$. Our assumption that $R_0 = \Cc$ implies that $\omega_\Rr$ is a pointed polyhedral cone, i.e., $\mathrm{null}(B) = \omega \cap (-\omega) = \{0\}$. In order to show that $\omega_m$ is finite, it is enough to show that $\omega_\Rr \cap (m - \omega_\Rr)$ is bounded.  We show this via contradiction.

     Suppose $\omega_\Rr \cap (m - \omega_\Rr)$ is unbounded. Then, there must exist $s \in \omega_\Rr \cap (m - \omega_\Rr)$ and a $v \in \Rr^k \setminus\{0\}$ so that for every $\lambda \geq 0$, $s + \lambda v \in \omega_\Rr \cap (m - \omega_\Rr)$. Since $s + \lambda v \in \omega_\Rr$, we have 
     \[
     B (s + \lambda v) \geq 0,
     \]
     and since $s + \lambda v \in m - \omega_\Rr$, we have
     \[
     B(m - s - \lambda v) \geq 0
     \]
     for all $\lambda \geq 0$. Solving each inequality for $\lambda Bv$ yields the following
     \[
     B(m-s) \geq \lambda Bv \geq -Bs
     \]
     for every $\lambda \geq 0$. Since $s,m,v,$ and $B$ are all fixed, the only way this holds true is if $Bv = 0$. This, however, contradicts that $\omega_\Rr$ is pointed; hence, $\omega_\Rr \cap (m - \omega_\Rr)$ is bounded.
\end{proof}

\begin{theorem}
    The procedure SortLatticePoints terminates and the output is a linear extension of the partial order $\preceq_\omega$.
\end{theorem}
\begin{proof}
    Since, at each step in the procedure, at least one element is sorted.
    Finiteness from Lemma~\ref{lemma:finite}
    yields that this procedure must
    terminate in finitely many steps.
    For the second claim, suppose we are at the step in the procedure where we are about to add $s$ to \textbf{Sorted}. The elements $t\in \omega$ which are covered by $s$ are all of the form $s - Ae_j$ for some $j \in \{1,\dotsc,N\}$. Therefore, by induction, when we sort $s$, we are guaranteeing that all elements less than $s$ in the partial order have already been sorted and that no elements greater than $s$ have been sorted yielding a linear extension.
\end{proof}

Our second procedure below computes $D_0^m(I)$ by utilizing the closedness subspace condition. The correctness of this procedure is the content of Cor.~\ref{cor: dSpace membership}
and illustrated in Ex.~\ref{ex:Hirzebruch3}.

\vspace{0.4cm}
\hrule
\begin{description}
    \item[Procedure $\mathrm{DualSpace}(m,I,A,B)$]
    \item[Input] 
    The matrix $A\in\Zz^{k\times N}$ where the $i^{\rm th}$ column is $\deg(x_i)$.  
    A lattice point $m \in \omega$. 
    A $\Zz^k$-graded ideal $I = \langle f_1,\dotsc,f_\ell\rangle$ with $\deg(f_i) = d_i$.
The matrix $B\in \Zz^{p\times k}$ with \mbox{$\omega_\Rr = \{y ~:~ By \geq 0\}$}.
    \item[Output] A basis for $D_0^m(I)$.
    \item[Begin] \hfill
    \begin{enumerate}
        \item Sort the lattice points in $\omega$ less than or equal to $m$, say $\{s_1,\dotsc,s_r\} := \mathrm{SortLatticePoints}(A,B,m)$ where $s_1 = 0$ and $s_r = m$.
        \item Set $C_0^0(I) := \mathrm{span}_\Cc \{\partial_1\}$ and $D_0^0(I) = C_0^0(I)$.
        \item For $i$ from 2 to $r$ do 
        \begin{enumerate}
            \item Compute a basis of $C_0^{s_i}(I) = \bigcap_{j=1}^N\Phi_j^{-1}(D_0^{s_i - Ae_j}(I))$.
            \item Impose the linear conditions
            that $\partial(f_j)=0$
            on the basis for 
            $C_0^{s_i}(I)$
            for $j=1,\dots,t$
            to compute a basis
            for 
            \[
            D_0^{s_i} = \{\partial \in C_0^{s_i}(I) ~:~ \partial(f_j) = 0 \text{ for } 1 \leq j \leq t\}
            \]
        \end{enumerate}
    \end{enumerate}
    \item[Return] a basis for $D_0^{s_r}(I)=D_0^{m}(I)$
\end{description}
\hrule

%
%

\section{Examples}\label{sec:Examples}

The following three examples
were computed using 
our {\tt Macaulay2} implementation 
described in Sec.~\ref{sec:Algorithm}.
Since our implementation is a proof-of-concept, 
it is not yet competitive with highly researched and optimized
Gr{\"o}bner basis methods.  However, as mentioned in the Introduction,
one advantage of a dual space approach is that one can start computing 
dual spaces immediately up to a given degree which would be particularly 
useful for problems in which computing a Gr{\"o}bner basis
is computationally more difficult than the following examples. 
See Sec.~\ref{sec:Conclusion} for comments regarding future research directions
including improved efficiency and incorporating parallel linear algebra routines. 

\subsection{Hirzebruch surface}

Examples~\ref{ex:Hirzebruch},
~\ref{ex:Hirzebruch2},
and~\ref{ex:Hirzebruch3} 
consider aspects of the Hirzebruch surface.
The following considers $\mathcal{H}_2$,
which is a smooth projective toric surface. The Cox ring 
is a polynomial ring $R = \Cc[x_1,x_2,x_3,x_4]$
which is graded by the Picard group,
namely $\Zz^2$. The degree of each $x_i$ is given by the equivalence class of $e_i$ in the cokernel of the transpose of 
    \[
        \begin{pmatrix}
            -1 & 0 & 1 &  0 \\
            2 & 1 & 0 & -1
        \end{pmatrix}.
    \]
After choosing a basis, 
$\deg(x_i) \in \Zz^2$ is given by the $i^\text{th}$ column of
    \[
        A = \begin{pmatrix}
            1  & 0 & 1 & 0 \\
            -2 & 1 & 0 & 1
            \end{pmatrix}
    \]
which corresponds with the $r=2$ case
in Ex.~\ref{ex:Hirzebruch2}.

Consider $f = x_1^2 x_2^6 + x_1^2 x_2^3 x_4^3 - x_3^2 x_4^2$
which is irreducible with $\deg f = (2,2)$.  Let $I = \langle f\rangle$.
By the toric ideal-variety 
correspondence~\cite[Prop.~5.2.4]{ToricVarietiesCLS}, $I$ cuts out an irreducible curve $C \subset \mathcal{H}_2$.

 The values from $(0,0)$ to $(4,4)$ of the multi-graded Hilbert function $H_I(i,j)$ are given in the table below. A dash is put in position $(i,j)$ if $(i,j) \succ (4,4)$ or $H_I(i,j) = 0$. 
    \[
    \hbox{\small
    \begin{tabular}{c|ccccc}
       \diagbox{$j$}{$i$} & 0 & 1 & 2 & 3 & 4 \\ \hline
       12&  13&-&-&-&-\\
       11&  12&-&-&-&-\\
       10&  11&24&-&-&-\\
       9&   10&22&-&-&-\\
       8&   9&20&26&-&-\\
       7&   8&18&24&-&-\\
       6&   7&16&22&28&-\\
       5&   6&14&20&26&-\\
       4&   5&12&18&24&30\\
       3&   4&10&16&22&28\\
       2&   3&8&14&20&26\\
       1&   2&6&12&18&24\\
       0&   1&4&9&15&21\\
       -1&  -&2&6&12&18\\
       -2&  -&1&4&9&15\\
       -3&  -&-&2&6&12\\
       -4&  -&-&1&4&9\\
       -5&  -&-&-&2&6\\
       -6&  -&-&-&1&4\\
       -7&  -&-&-&-&2\\
       -8&  -&-&-&-&1
       \end{tabular}}
       \]

Since the class of $(1,1)$ in $\mathrm{Pic}(\mathcal{H}_2)$ is very ample, we can embed~$C$ in $\Pp^5$ 
via this divisor. 
By looking at the values $H_I(a,a)$ for $a \geq 2$, we see that the Hilbert polynomial of $C \subseteq \Pp^5$ is $8a - 2$ from which we conclude that $C$ has degree 8 and arithmetic genus 3.

\subsection{Parameter geography}

In \cite{ParameterGeography}, the authors study the following 
parameterized system $\Phi_1(u,v;\sigma)= \Phi_2(u,v;\sigma) = 0$
where 
$\theta_1,\dotsc,\theta_8$
are taking to be generic and $\zeta = 1$:
$$\begin{array}{c}
\hbox{\small $\Phi_1(u,v;\sigma)$}= 
\\
\hbox{\small $
\theta_1 v^2 + \zeta uv + \theta_2 \zeta^2 u^2 + (\theta_1\theta_3 - \theta_1\theta_3\sigma - \theta_1 + \theta_7 \zeta) uv^2$}\\
\hbox{\small $+(\theta_4\zeta - \theta_4\zeta\sigma - \zeta + \theta_2\theta_8\zeta^2)u^2v $}\\
\hbox{\small $+(\theta_2\theta_5\zeta^2 - \theta_2\theta_5\zeta^2\sigma - \theta_2\zeta^2)u^3 + \theta_1\theta_6v^3 -(\theta_1\theta_3 + \theta_7\zeta)u^2v^2 $}\\
\hbox{\small $- (\theta_4\zeta + \theta_2\theta_8\zeta^2)u^3v-\theta_2\theta_5\zeta^2u^4 - \theta_1\theta_6 uv^3$} \\[0.01in]
\hbox{\small $\Phi_2(u,v;\sigma) = $}\\
\hbox{\small $\theta_1v^2 + \zeta uv + \theta_2\zeta^2u^2 + (\theta_1\theta_6 - \theta_1\theta_6 - \theta_1)v^3 $}\\
\hbox{\small $+(\theta_7\zeta - \theta_7\zeta\sigma - \zeta + \theta_1\theta_3)uv^2 $}\\
\hbox{\small $+(\theta_2\theta_8\zeta^2 - \theta_2\theta_8\zeta^2\sigma - \theta_2\zeta^2 + \theta_4\zeta)u^2v + \theta_2\theta_5\zeta^2u^3 $}\\
\hbox{\small $-(\theta_1\theta_3 + \theta_7\zeta)uv^3 -(\theta_4\zeta + \theta_2\theta_8\zeta^2)u^2v^2 
-\theta_2\theta_5\zeta^2u^3v - \theta_1\theta_6v^4$} \end{array}$$
Consider homogenizing by adding $\tau$ and $w$ to consider the polynomial ring $\Cc[\sigma,\tau, u,v,w]$ where $\deg \sigma = \deg \tau = (1,0)$ and $\deg u = \deg v = \deg w = (0,1)$.
This yields a $\Zz^2$-graded ideal with 2 generators and we view the zero locus of this system as a reducible curve in $\Pp^1 \times \Pp^2$. 
After slicing this system with a generic linear form of degree $(0,1)$, 
the following table lists multi-graded Hilbert function up to $(10,10)$ computed via Macaulay dual spaces. 
  \[
  \hbox{\small
\begin{tabular}{c|ccccccccccc}
    \diagbox{$i$}{$j$} &0 & 1 & 2 &3&4&5&6&7&8&9&10 \\ \hline
      0&1&2&3&4&5&6&7&8&9&10&11\\
      1&2&4&6&8&8&8&8&8&8&8&8\\
      2&3&6&9&12&11&10&9&8&8&8&8\\
      3&4&8&12&16&14&12&10&8&8&8&8\\
      4&5&10&15&20&17&14&11&8&8&8&8\\
      5&6&12&18&24&20&16&12&8&8&8&8\\
      6&7&14&21&28&23&18&13&8&8&8&8\\
      7&8&16&24&32&26&20&14&8&8&8&8\\
      8&9&18&27&36&29&22&15&8&8&8&8\\
      9&10&20&30&40&32&24&16&8&8&8&8\\
      10&11&22&33&44&35&26&17&8&8&8&8
\end{tabular}}
 \]

Using Macaulay dual spaces,
we saturated away the components
lying along coordinate axes
resulting in the following
multi-graded Hilbert function.
\[\hbox{\small
\begin{tabular}{c|ccccccc}
\diagbox{$i$}{$j$}&0&1&2&3&4&5&6\\ \hline
      0&1&2&3&4&5&6&7\\
      1&2&4&6&7&7&7&7\\
      2&3&6&8&7&7&7&7\\
      3&4&8&9&7&7&7&7\\
      4&5&10&10&7&7&7&7\\
      5&6&12&11&7&7&7&7\\
      6&7&14&12&7&7&7&7
\end{tabular}}
\]
Although this table provides information
when viewed as a subvariety of $\Pp^1 \times \Pp^2$,
the Hilbert function of the system viewed in $\Pp^4$ via a Segre product 
is exactly $H_I(i,i)$. 
Hence, since the values along the main diagonal stabilize at $7$, 
there are 7 non-zero complex solutions to this system for a generic choice of $\sigma$ matching the results in \cite{ParameterGeography}.

\subsection{Chemical reaction network}

The final example considers
a chemical reaction network known as the one-site phosphorylation cycle \cite{GrossElizabeth2021Tsda}. The \textit{steady-state degree} of this chemical reaction network is the number of complex solutions to the following system for generic parameters $c_A$ and $k_{ij}$.
\vspace{-0.05in}
\begin{align*}
    \hbox{\small $f_1$}&\hbox{\small $= x_E + x_{X_1} - c_E - c_{X_1}$}\\
    \hbox{\small $f_2$}&\hbox{\small $= x_F + x_{Y_1} - c_F - c_{Y_1}$}\\
    \hbox{\small $f_3$}&\hbox{\small $= x_{S_0} + x_{S_1} - x_E - x_F - c_{S_0} - c_{S_1} + c_E + c_F$}\\
    \hbox{\small $f_4$}&\hbox{\small $= -k_{01}x_{S_0} x_E + k_{10}x_{X_1} + k_{45}x_{Y_1}$}\\
    \hbox{\small $f_5$}&\hbox{\small $= -k_{34}x_{S_1} x_F + k_{12}x_{X_1} + k_{43}x_{Y_1}$}\\
    \hbox{\small $f_6$}&\hbox{\small $= k_{01}x_{S_0} x_E - (k_{10} + k_{12})x_{X_1}$}\\
    \hbox{\small $f_7$}&\hbox{\small $= k_{34}x_{S_1} x_F - (k_{43} + k_{45})x_{Y_1}$}\\[-0.2in]
\end{align*}
One way to compute the steady-state 
degree is to homogenize 
with respect to a new variable $t$,
saturate away the hyperplane at infinity,
and compute the degree of the resulting
projective variety. 
Letting~$I$ be the ideal generated
by the homogenization with respect to $t$ of $f_1,\dots,f_7$,
one obtains the following
using Macaulay dual spaces.
\[\hbox{\small
    \begin{tabular}{c|ccccccccccc}
        $k$           & 0 & 1 & 2 & 3 & 4 & 5 & 6 & 7 & 8 & 9 & 10 \\ \hline
        $H_I(k)$      & 1 & 4 & 7 & 8 & 8 & 8 & 8 & 8 & 8 & 8 & 8 \\
        $H_{I:t}(k)$  & 1 & 3 & 4 & 4 & 4 & 4 & 4 & 4 & 4 & 4 & - \\
        $H_{I:t^2}(k)$& 1 & 3 & 3 & 3 & 3 & 3 & 3 & 3 & 3 & - & - \\
        $H_{I:t^3}(k)$& 1 & 3 & 3 & 3 & 3 & 3 & 3 & 3 & - & - & -
    \end{tabular}
}\]
Since $H_{I:t^2} = H_{I:t^3}$, we can conclude that $H_{I:t^2} = H_{I:t^\infty}$. 
Hence, this computation shows
the steady-state degree is $3$
in agreement with the results found in \cite{GrossElizabeth2021Tsda}.

\section{Conclusion}\label{sec:Conclusion}

Building on a key theoretical contribution in Thm.~\ref{thm:MhomDualBasis} which
shows that the Macaulay dual space of a multi-graded ideal
is multi-graded, algorithms are presented
for performing computations related to
such dual spaces
including using Thm.~\ref{thm: saturation}
which describes how to compute ideal quotients using dual spaces.
Using a proof-of-concept implementation
in {\tt Macaulay2}~\cite{M2},
ideal computations were performed
using multi-graded dual spaces on several different examples.

Some future research directions include
incorporating more efficient and parallel numerical linear algebra
routines into the implementation to improve the performance,
consider examples where obtainig a Gr{\"o}bner basis is computationally
more challenging, consider errors and stability when performing
numerical linear algebra routines with dual spaces,
and investigate the complexity of performing
computations using dual~spaces.

\nocite{*}
\bibliographystyle{acm}
\bibliography{ref}

\end{document}